\documentclass[a4paper,12pt]{amsart}

\usepackage{a4wide}
\usepackage{latexsym,amssymb,amsmath,tikz, graphicx}
\usepackage{enumitem}
\newtheorem{lemma}{Lemma}
\newtheorem{proposition}[lemma]{Proposition}
\newtheorem{theorem}[lemma]{Theorem}

\newtheorem{dmydefs}[lemma]{Definition}

\newtheorem{problem}{Problem}

\begin{document}

\title{An infinite antichain of planar tanglegrams
}


\author{\'Eva Czabarka}
\address{\'E. Czabarka\\
              Department of Mathematics\\
              University of South Carolina, Columbia SC, USA \\
              Visiting Professor, Department of Mathematics and Applied Mathematics\\
              University of Johannesburg, South Africa}
              \email{czabarka@math.sc.edu}
              
\author{Stephen J. Smith}
\address{ S.J. Smith \\
               Department of Mathematics\\
              University of South Carolina, Columbia SC, USA}           
               \email{sjs8@email.sc.edu}
\author{ L\'aszl\'o A. Sz\'ekely}
\address{L.A. Sz\'ekely\\
               Department of Mathematics\\
              University of South Carolina, Columbia SC, USA \\
              Visiting Professor, Department of Mathematics and Applied Mathematics\\
              University of Johannesburg, South Africa}
              \email{szekely@math.sc.edu}           

\thanks{The last  author was  supported in part by the National Science Foundation  contract  DMS  1600811.}
\keywords{binary tree, caterpillar, induced subtree, tanglegram, planar tanglegram, induced subtanglegram, permutation, permutation pattern, partial order, well-quasi-ordering, antichain}
 \subjclass{Primary 05C05; secondary 05C10, 06A06, 92B10}

\begin{abstract}
Contrary to the expectation arising from  the tanglegram Kuratowski theorem of \'E.~Czabarka,  L.~A. Sz\'ekely and  S.~Wagner
[{\textit SIAM J. Discrete Math.}  31(3): 1732--1750, (2017)], we construct an infinite antichain of planar tanglegrams with respect to the induced subtanglegram partial order. 
R.E. Tarjan, R. Laver,
D.A. Spielman and M. B\'ona, and possibly others,
showed that  the partially ordered set of finite permutations ordered by deletion of entries contains an infinite antichain,
 i.e. there exists an infinite collection of permutations, such that  none of them contains another as a pattern. Our construction adds a twist to the construction
 of Spielman and B\'ona [{\textit Electr. J. Comb.} Vol. 7. N2.]
\end{abstract}

\maketitle

\section{Introduction}
\label{intro}
Informally, a
{\textit tanglegram} is a specific kind of graph, consisting of two rooted binary trees of the same size and a perfect matching joining their leaves. Tanglegrams are drawn 
under specific rules, such drawings are called {\textit tanglegram layouts}. (Formal definitions are postponed to Section~\ref{sec:defs}.) The {\textit tangle crossing number} of a tanglegram is the minimum crossing number (i.e. the minimum number of unordered crossing edge-pairs) among its layouts. The tanglegram is {\textit planar}, if it has a layout without crossings.  
Tanglegrams play a major role in phylogenetics, especially in the theory of cospeciation \cite{page}. The first binary tree is the phylogenetic tree of hosts, while the second binary tree is the phylogenetic tree of their parasites, e.g. gopher and louse 
\cite{HafnerNadler}. The matching connects the host with its parasite.
 The tanglegram  crossing number  has been related to the  number of times  parasites switched hosts \cite{HafnerNadler},
or, working with gene trees instead of phylogenetic trees, to the 
number of horizontal gene transfers (\cite{Burt}, pp. 204--206).
Tanglegrams are well-studied objects in phylogenetics and computer science (see e.g. 
\cite{bansal,billey,buchin,calamoneri,induci,fernau,Henzinger,KonvWag,Matsen,scornavacca}).

Czabarka, Sz\'ekely and Wagner \cite{tanglekurat} discovered a Kuratowski-type theorem that characterized planar tanglegrams by two excluded induced subtanglegrams. They asked 
\begin{problem} \label{open} Are there similar characterizations
\begin{enumerate}[label={\upshape (\roman*)}]
\item\label{part:first} for tanglegrams with tangle crossing number at most $k$? 
\item  for tanglegrams that have a layout without $k$ pairwise crossing edges? 
\end{enumerate}
\end{problem}
Were the induced subtanglegram partial order a well-quasi-ordering, the answer to 
these questions would immediately be  in the affirmative, delivering a number of algorithmic consequences. To be a well-quasi-ordering, there should not be an infinite antichain 
in the well-founded partially order.

Whether   a well-founded partially ordered set  has an infinite antichain has been well studied 
(e.g. \cite{higman,kruskal,robsey1,robsey2}). In particular, Kruskal's Tree Theorem \cite{kruskal} would give one hope that the induced subtanglegram relation would be a 
well-quasi-ordering as well. However, tanglegrams, where the two trees are caterpillars,
are closely related  to permutations and permutation patterns (see Section~\ref{sec:defs}).
Laver~\cite{laver}, Pratt~\cite{pratt}, Tarjan~\cite{tarjan}, and Speilman and B\'ona~\cite{bona}  constructed infinite antichains of permutations 
for the partial order defined by permutation patterns.

While the antichain of permutations in  \cite{bona}  does not immediately yield an infinite antichain of tanglegrams (in fact, it defines a chain, as will be explained at the end of Section~\ref{sec:construct}), when we turn these permutations ``upside down" (i.e. in a permutation of $[n]$ we replace every entry $j$ by $n+1-j$), we manage to obtain an infinite antichain of tanglegrams with respect to the induced tanglegram relation. Furthermore, the elements of the antichain are planar tanglegrams
(shown in Section~\ref{sec:planar}), 
making Problem~\ref{open} even more intriguing. An algorithmic consequence of a positive answer to Problem~\ref{open}~\ref{part:first} would be fixed-parameter tractability of computing the
tanglegram crossing number, a result that is already known \cite{buchin}. 

The authors wish to thank Stephan Wagner and Mikl\'os B\'ona  for helpful discussions.


\section{Definitions and basic setup}\label{sec:defs}

As customary,  $[n]$ denotes the set $\{1,2,3,\ldots,n\}$, $S_n$ denotes the symmetric group acting on $[n]$. For $\pi\in S_n$, we  use the notation $\pi=(a_1,\ldots,a_n)$, if $\pi(i)=a_i$ for all $i\in[n]$.

\begin{dmydefs} A  {\textit rooted tree} $T$ is a tree with a distinguished vertex called the {\textit root}. Given a vertex $v$ in a rooted tree, and a neighbor $y$ of $v$, $y$ is the {\textit parent} of $v$,  if $y$ is on
the path from $v$ to the root; otherwise $y$ is a {\textit child} of $v$. The rooted tree $T$ is {\textit binary}, if every vertex has zero or two children.
\end{dmydefs}

\begin{dmydefs} 
For $n\ge 2$, the {\textit rooted caterpillar} $C_n$ with $n$ leaves is the rooted binary tree, whose $n-2$ internal vertices form a path, and the root is an endvertex of this path.
\end{dmydefs}

Note that $C_n$ has two leaves at distance $n-1$ from the root, and for all $i$ $( 1\le i\le n-2)$  it has  precisely one leaf at distance $i$ from the root. These properties 
characterize $C_n$.

\begin{dmydefs} Given a  rooted binary tree $T$ with root $r$ and a non-empty subset $B$ of its leaves, the {\textit rooted binary subtree induced by $B$},  $T[B]$, is obtained as follows:
Take the smallest subtree $T'$  of $T$ containing all vertices of $B$, and designate the vertex $\rho\in V(T')$ closest to $r$ in $T$ as the root of $T'$. 
This  rooted tree is not necessarily  binary---suppress  all 
 vertices  of     degree $2$    (except  $\rho$) in $T'$ to make it binary. The resulting rooted binary tree is  $T[B]$.
\end{dmydefs}

\begin{dmydefs} A {\textit tanglegram} of size $n$ is an ordered triplet $(T_1,T_2,M)$, where $T_1$ and $T_2$ are rooted binary trees with $n$ leaves each, and $M$ is a perfect matching between the two leaf sets. $T_1$ is  called the {\textit left tree} and $T_2$ is the {\textit right tree} of the tanglegram. Two tanglegrams are considered   the same, if there is a graph isomorphism between them, which  fixes the roots of the left tree and the right tree.
\end{dmydefs}

\begin{dmydefs} Given a tanglegram $\mathcal{T}=(T_1,T_2,M)$ and an $\emptyset\not= M'\subseteq M$, the {\textit subtanglegram induced by $M'$} is
$\mathcal{T}[M']=(T_1[B_1],T_2[B_2],M')$, where $B_i$ is the set of leaves in $T_i$ matched by $M'$. We say that $\mathcal{T}^*$ is an {\textit induced subtanglegram of $\mathcal{T}$} (in notation: $\mathcal{T}^*\preceq\mathcal{T}$), if there is an $M^*\subseteq M$ such that $\mathcal{T}^*=\mathcal{T}[M^*]$.
\end{dmydefs}
Note that $\preceq$ is a partial order on the set of tanglegrams, and $\preceq$ is well-founded, i.e. it has no infinite strictly decreasing chains.

\begin{dmydefs}
Given a tanglegram $\mathcal{T}=(T_1,T_2,M)$, where the root of $T_i$ is $r_i$, the  \textit{multiset of distance pairs}, $\mathbb{D}(\mathcal{T})$, contains exactly $k$ copies of $(d_1,d_2)$ if and only if there exists exactly $k$ matching
edges of the form $(x_1,x_2)\in M$ such that $x_i$ is a leaf of $T_i$ at distance $d_i$ from $r_i$.
\end{dmydefs}

From now on  we restrict ourselves to tanglegrams, in which both the left and right trees are rooted caterpillars.
Note that in this case, if two tanglegrams have the same distance  pair multiset, then they are the same.

\begin{dmydefs} \label{convention}
For $n\ge 2$, the {\textit distance labeling} of the leaves of $C_n$ is the following:  for each $i$, $1\le i\le n-2$, the leaf labeled $i$ is the one at distance $i$ from the root, and  the two leaves at distance $n-1$ are labeled arbitrarily by $n-1$ and $n$.\\
For $n\ge 2$ and  $\pi\in S_n$, the {\textit catergram} $\mathcal{T}_{\pi}$ is the tanglegram $(C_n,C_n,M_{\pi})$, where $M_{\pi}$ is defined as follows.
Use the distance labeling of the leaves of both caterpillars,
match the leaf on the left tree labeled $i$  with the leaf on the right tree labeled $j$ if and only if $\pi(i)=j$.
\end{dmydefs}

Note that every tanglegram, in which both the  left  tree and right tree are rooted caterpillars, does arise as a catergram, but the permutation that defines it is not unique.

\begin{dmydefs} Assume $n\ge 2$. Given a $\pi=(a_1,\ldots,a_n)\in S_n$, we  define the (not necessarily different) permutations $\widehat{\pi}$, $\widetilde{\pi}$ as
\begin{equation*} \widehat{\pi}(i)=
	\begin{cases} 
		a_i, & \hbox{ if }  i\le n-2 \\
		a_n, & \hbox{ if }  i=n-1\\
      		a_{n-1}, & \hbox{ if }   i=n
   	\end{cases}\text{\ \ \ and\ \ \ }
\widetilde{\pi}(i)=
	\begin{cases} 
		a_i, & \hbox{ if } a_i\notin\{n-1,n\} \\
      		n-1, & \hbox{ if } a_i=n\\
		n,  & \hbox{ if } a_i=n-1;
   	\end{cases},	
\end{equation*}
and finally let  $\pi^*=\widehat{(\widetilde{\pi})}$. We define the set $\overline{\pi}=\{\pi,\widehat{\pi},\widetilde{\pi},\pi^*\}$.
\end{dmydefs}

\begin{proposition}\label{claim:hattilde} The following facts are obvious for any $\pi=(a_1,\ldots,a_n)$: 
\begin{enumerate}[label={\upshape (\alph*)}]
\item We have  $\mathbb{D}(\mathcal{T}_{\pi})=\{(1, a_1^*), (2, a_2^*),\dots, (n-1, a_{n-1}^*), (n-1, a_n^*)\}$, where
\begin{equation*}a_i^*=
	\begin{cases} 
		a_i, & a_i<n \\
      		n-1, & a_i=n.
   	\end{cases}
\end{equation*}
\item  $\widehat{(\widetilde{\pi})}=\widetilde{(\widehat{\pi})}$, $\pi=\widehat{(\widehat{\pi})}=\widetilde{(\widetilde{\pi})}$,
 and $\pi\notin\{\widehat{\pi}, \widetilde{\pi}\}$.
 \item  $\rho\in \overline{\pi}$ iff $\overline{\rho}=\overline{\pi}$.
\item $\widehat{\pi}=\widetilde{\pi}$ iff $\{a_{n-1},a_n\}=\{n-1,n\}$ iff $\pi=\pi^*$; consequently $|\overline{\pi}|\in\{2,4\}.$
\item\label{part:equal} $\mathcal{T}_{\rho}=\mathcal{T}_{\pi}$ iff $\mathbb{D}(\mathcal{T}_{\rho})=\mathbb{D}(\mathcal{T}_{\pi})$ iff $\rho\in\overline{\pi}$.
\end{enumerate}
\end{proposition}

\begin{dmydefs} We say that two sequences of $n$ numbers, $(a_1,\ldots,a_n),(b_1,\ldots,b_n)\in\mathbb{R}^n$, are {\textit order isomorphic,} if for all $i,j\in[n]$, we have  $a_i<a_j$ iff $b_i<b_j$. Given a $\pi\in S_n$ and a non-empty $A\subseteq [n]$, where $a_1,\ldots,a_k$ lists the elements of $A$ in increasing order, we denote by $\pi[A]$ the permutation in $S_{|A|}$ that is order isomorphic to $(\pi(a_1),\pi(a_2),\ldots,\pi(a_k))$. If $\rho\in S_m$ and $\pi\in S_n$, then we say that 
$\rho$ is a {\textit pattern in $\pi$} (in notation $\rho\leq \pi$), if  $\pi[A]=\rho$ for some $A\subseteq [n]$.
\end{dmydefs}

\begin{dmydefs}  Assume  $\pi\in S_n$ and $\emptyset\not= A\subseteq[n]$. Then (with a slight abuse of notation) we denote by $\mathcal{T}_{\pi}[A]$ the induced subtanglegram $\mathcal{T}_{\pi}[M^*]$, where
$M^{*}$ is the matching containing edges of $M$ incident upon leaves of the left tree that are labeled with elements of $A$.
\end{dmydefs}

\begin{proposition}\label{claim:main} The following statements are true:
\begin{enumerate}[label={\upshape (\alph*)}]
\item\label{part:cater} Let $v$ be a leaf of $C_n$ at distance $i$ from the root  $r$ of $C_n$, and $y\ne v$ be another leaf that is at distance $j$ from $r$. Let $T$ be the binary tree induced by all leaves except $v$ (so $T=C_{n-1}$) with root $r^*$. Then $y$ is a leaf in $T$, and the distance of $y$ from $r^*$ is $j$ if $j<i$, and $j-1$ otherwise.
\item\label{part:exchange} For any $\pi\in S_n$ and non-empty $A\subseteq[n]$, we have $\mathcal{T}_{\pi}[A]=\mathcal{T}_{\pi[A]}$. \emph {(This follows from~\ref{part:cater}).}
\item\label{part:goal} For $\rho\in S_m$ and $\pi\in S_n$, we have $\mathcal{T}_{\rho}\preceq\mathcal{T}_{\pi}$ iff $\mathcal{T}_{\rho}=\mathcal{T}_{\pi[A]}$ for some $A\subseteq[n]$  iff $\sigma\leq \pi$  for some $\sigma\in\overline{\rho}$. \emph{ (This follows from~\ref{part:exchange} and 
Proposition~\ref{claim:hattilde}~\ref{part:equal}).}
\end{enumerate}
\end{proposition}

\section{Constructing the antichain of tanglegrams}\label{sec:construct}

\begin{dmydefs}  For $i\in\mathbb{Z}^+$, we set $\rho_i\in S_{[12+2i]}$ as
$(\rho_i(1),\rho_i(2),\rho_i(3),\rho_i(4))=(2,3,5,1)$, $(\rho_i(9+2i),\rho_i(10+2i),\rho_i(11+2i),\rho_i(12+2i))=(10+2i,11+2i,12+2i,8+2i)$ and for
$j:5\le j\le 8+2i$
\begin{equation*}
\rho_i(j)=	\begin{cases} 
		j+2, & \text{ if $j$ is odd} \\
      		j-2, &\text{ if $j$ is even}.
   	\end{cases}	
\end{equation*}
\end{dmydefs}

So for example, the first two permutations  in our sequence will be
\begin{eqnarray*}
\rho_1&=&(2,3,5,1,7,4,9,6,11,8,12,13,14,10)\\
\rho_2&=&(2,3,5,1,7,4,9,6,11,8,13,10,14,15,16,12).
\end{eqnarray*}

Spielman and B\'ona \cite{bona} showed that  if $\pi_i$ is $\rho_i$ turned  ``upside down", then $\{\pi_i: i\in\mathbb{Z}^+\}$ is an antichain for the pattern partial order of permutations. 
We are now ready to show our result:
\begin{theorem} $\{\mathcal{T}_{\rho_i}:i\in\mathbb{Z}^+\}$ is an antichain with respect to the relation $\preceq$.
\end{theorem}

\begin{proof}
In the proof we will use the fact that for any $k$ and 
any $\gamma\in\overline{\rho_k}$, the permutation $\gamma$ has exactly two entries that are preceded by at least $3$ larger elements:
the entry $1$ and the entry $8+2i$; moreover, if $\gamma\in\{\rho_k,\widehat{\rho_k}\}$ then $8+2k$ is preceded by exactly $4$ larger elements, but these $4$ elements are not order isomorphic in $\rho_k$ and $\widetilde{\rho_k}$.

By Proposition~\ref{claim:main}~\ref{part:goal}, it is sufficient to show that for any $i<j$ and for any $\sigma\in\overline{\rho_i}$, $\sigma\not\leq\rho_j$.
By our starting remark, if $\sigma<\rho_j$, then  the entries $1$ and $8+2i$ in $\sigma$   should map to the entries $1$ and $8+2j$ in $\rho_j$, and the preceding larger elements must map to preceding larger entries; consequently $\widetilde{\rho_i}\not\leq\rho_j$.
 As $8+2j$ is the last entry of $\rho_j$, but not of $\widehat{\rho_i}$ or $\rho_i^*$ (unless $\rho_i^*=\rho_i$), we get that $\widehat{\rho_i}\not\leq\rho_j$ and
$\rho_i^*\not\leq\rho_j$. So what remains to be shown is $\rho_i\not\leq \rho_j$, which was essentially stated and proved in \cite{bona}, but for completeness, we include a (somewhat different) proof here.

Suppose  for contrary  that $\rho_i < \rho_j$, i.e. entries of $\rho_i$ map  to  entries of $\rho_j$ in an order preserving fashion. 
By our earlier remarks, the first $4$ elements of $\rho_i$ must map to the first $4$ elements of $\rho_j$ and the last $6$ elements of
$\rho_i$ must map to the last $6$ elements of $\rho_j$, so me must map the sequence
$(7,4,9,6,\ldots,7+2i,4+2i$) to $(7,4,9,6,\ldots,7+2j,4+2j)$ by leaving out $2(j-i)\ge 2$ elements.

Let $x$ be an entry of the contiguous subsequence  $(7,4,9,6,\ldots,7+2k,4+2k$) of $\rho_k$.
If $x$ is even, then there are no entries that appear after $x$ in $\rho_k$ that are smaller than $x$, and $x$ is preceeded by the entry $x+1$.
If $x$ is odd, then there are exactly two entries in $\rho_k$ that follow $x$ and are smaller than $x$, and they are both even.

Let $x$ now be the first entry that is erased from $\rho_j$. The entries before $x$ in $\rho_i$ are mapped to the same entries, respectively, in $\rho_j$, and the entry
$x$ in $\rho_i$ is mapped to a different entry that appears after $x$ in $\rho_j$.

If $x$ is even, then, as the entry $x+1$ is before $x$ in $\rho_i$, $x$ must map to an entry smaller than $x+1$ but is after $x$ in $\rho_j$. As such an entry does not exist, $x$ must be odd.

As $x$ is odd, it is immediately followed by the even entry $x-3$ in both $\rho_i$ and $\rho_j$, and preceeded by the entry $x-2$, which was not erased from $\rho_j$. As  entry $x-2$ in $\rho_i$ maps to  entry $x-2$ in $\rho_j$, and entry $x$ in $\rho_i$ maps to an entry after $x$ in $\rho_j$, it follows that entry $x-3$ in $\rho_i$ must map to an entry that is after $x-3$ in $\rho_j$ and is smaller than $x-3$. Since
such an entry does not exist, $\rho_i\not\leq\rho_j$. \qed
\end{proof}

We remark here that in the  infinite antichain of permutations   $\{\pi_i:i\in\mathbb{Z}^+\}$ of \cite{bona}, $\pi_i$ is  our $\rho_i$ is  {\textit turned ``upside down"}. For example,
\begin{eqnarray*}
\pi_1&=&(13,12,10,14,8,11,6,9,4,7,3,2,1,5)\\
\pi_2&=&(15,14,12,16,10,13,8,11,6,9,4,7,3,2,1,5).
\end{eqnarray*}
One can easily check that for $A=[16]\setminus\{2,4\}$ we get $\widetilde{\pi_1}=\pi_2[A]$, showing that $\mathcal{T}_{\pi_1}\preceq\mathcal{T}_{\pi_2}$.
Moreover, for every $i\in\mathbb{Z}^+$, setting $A_i=[14+2i]\setminus\{2,4\}$,  we observe that $\widetilde{\pi_i}=\pi_{i+1}[A_i]$, showing that 
\begin{proposition}
  $\{\mathcal{T}_{\pi_i}:i\in\mathbb{Z}^+\}$ is  an infinite chain in the induced subtanglegram partial order.
\end{proposition}
This is why we had to put a twist on the construction of \cite{bona}.

\section{Planarity of the tanglegrams in the antichain}\label{sec:planar}

Lastly, we   show that the tanglegrams $\mathcal{T}_{\rho_i}$ are planar. For this we need to define layouts first.

\begin{dmydefs}
A {\em plane binary tree} is a rooted binary tree, in which the  children of  internal vertices are specified as left and right children. A plane binary tree is easy to draw  on one side 
of a line, without edge crossings, such that only the leaves of the tree are on the line. We will say that the plane binary tree $P$ is a {\textit plane tree of the rooted
binary tree $T$}, if $P$ is isomorphic to $T$ as a graph.
\end{dmydefs}

Note that if we label all vertices of a rooted binary tree with $n$ leaves, 
 then there are $2^{n-1}$ labeled plane trees whose underlying labeled graph is this labeled rooted binary tree.

\begin{dmydefs}
A  {\em  layout} $(L,R,M)$ of the tanglegram $\mathcal{T}=(T_1,T_2,M)$ is given by  a left plane binary tree $L$ isomorphic to $T_1$, drawn in the halfplane $x\leq 0$, having
its leaves on the  line $x=0$,
a right plane binary tree $R$ isomorphic to $T_2$  drawn in the halfplane $x\geq 1$, having
its leaves on the  line $x=1$,    
and the perfect matching  $M$ between their leaves drawn in straight line segments. \emph{(See Figure~\ref{fig:layout}.)}
\end{dmydefs}

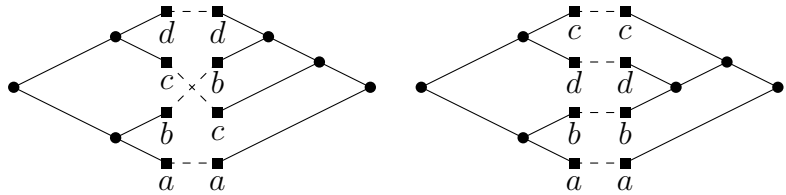
\begin{figure}[htbp]
\begin{center}
\begin{tikzpicture}[scale=.67]
        \node[fill=black,circle,inner sep=1.5pt]  at (-1,0) {}; 
        \node[fill=black,circle,inner sep=1.5pt]  at (1,-1) {};
        \node[fill=black,circle,inner sep=1.5pt]  at (1,1) {};
        \node[fill=black,rectangle,inner sep=2pt]  at (2,-1.5) {};
        \node[fill=black,rectangle,inner sep=2pt]  at (2,-.5) {};
        \node[fill=black,rectangle,inner sep=2pt]  at (2,.5) {};
        \node[fill=black,rectangle,inner sep=2pt]  at (2,1.5) {};

	\draw (-1,0)--(2,1.5);
	\draw (-1,0)--(2,-1.5);
	\draw (1,-1)--(2,-.5);
	\draw (1,1)--(2,.5);

        \node[fill=black,rectangle,inner sep=2pt]  at (3,-1.5) {};
        \node[fill=black,rectangle,inner sep=2pt]  at (3,-.5) {};
        \node[fill=black,rectangle,inner sep=2pt]  at (3,.5) {};
        \node[fill=black,rectangle,inner sep=2pt]  at (3,1.5) {};
        \node[fill=black,circle,inner sep=1.5pt]  at (4,1) {};
        \node[fill=black,circle,inner sep=1.5pt]  at (5,.5) {};
        \node[fill=black,circle,inner sep=1.5pt]  at (6,0) {}; 

	\draw (6,0)--(3,1.5);
	\draw (6,0)--(3,-1.5);
	\draw (5,.5)--(3,-.5);
	\draw (4,1)--(3,.5);

	\draw [dashed] (2,-1.5)--(3,-1.5);
	\draw [dashed] (2,-.5)--(3,.5);
	\draw [dashed] (3,-.5)--(2,.5);
	\draw [dashed] (2,1.5)--(3,1.5);
	
	\node at (2,-1.9) {$a$};
	\node at (3,-1.9) {$a$};
         \node at (2,-.9)  {$b$};
         \node at (3,.1)  {$b$};
         \node at (3,-.9)  {$c$};
         \node at (2,.1)  {$c$};
         \node at  (2,1.1) {$d$};
         \node at (3,1.1)    {$d$};
         
        \node[fill=black,circle,inner sep=1.5pt]  at (7,0) {}; 
        \node[fill=black,circle,inner sep=1.5pt]  at (9,-1) {};
        \node[fill=black,circle,inner sep=1.5pt]  at (9,1) {};
        \node[fill=black,rectangle,inner sep=2pt]  at (10,-1.5) {};
        \node[fill=black,rectangle,inner sep=2pt]  at (10,-.5) {};
        \node[fill=black,rectangle,inner sep=2pt]  at (10,.5) {};
        \node[fill=black,rectangle,inner sep=2pt]  at (10,1.5) {};

	\draw (7,0)--(10,1.5);
	\draw (7,0)--(10,-1.5);
	\draw (9,-1)--(10,-.5);
	\draw (9,1)--(10,.5);

        \node[fill=black,rectangle,inner sep=2pt]  at (11,-1.5) {};
        \node[fill=black,rectangle,inner sep=2pt]  at (11,-.5) {};
        \node[fill=black,rectangle,inner sep=2pt]  at (11,.5) {};
        \node[fill=black,rectangle,inner sep=2pt]  at (11,1.5) {};
        \node[fill=black,circle,inner sep=1.5pt]  at (12,0) {}; 
        \node[fill=black,circle,inner sep=1.5pt]  at (13,.5) {}; 
        \node[fill=black,circle,inner sep=1.5pt]  at (14,0) {}; 

	\draw (14,0)--(11,1.5);
	\draw (14,0)--(11,-1.5);
	\draw (13,.5)--(11,-.5);
	\draw (11,.5)--(12,0);

	\draw [dashed] (10,-1.5)--(11,-1.5);
	\draw [dashed] (10,-.5)--(11,-.5);
	\draw [dashed] (10,.5)--(11,.5);
	\draw [dashed] (10,1.5)--(11,1.5);
	
	\node at (10,-1.9) {$a$};
	\node at (10,-.9) {$b$};
	\node at (10,.1) {$d$};
	\node at (10,1.1) {$c$};
	\node at (11,1.1) {$c$};
	\node at (11,0.1) {$d$};
	\node at (11,-.9) {$b$};
	\node at (11,-1.9) {$a$};
\end{tikzpicture}
\end{center}
\caption{Two layouts of the same tanglegram. The leaf labels help showing that the two tanglegrams are identical.
}  \label{fig:layout}
\end{figure}

\begin{dmydefs}
A tanglegram  is {\textit planar} if it has a layout without crossing edges.
\end{dmydefs}

\begin{theorem}[Czabarka, Sz\'ekely, Wagner \cite{tanglekurat}] \label{thm:kurat}
Every non-planar tanglegram contains one of the  two tanglegrams in \emph{Figure~\ref{fig:excluded}} as
an induced subtanglegram. 
\end{theorem}

\begin{figure}[htbp]
\begin{center}
\begin{tikzpicture}[line/.style={-},scale=.7]
  \node(A)[fill=black,circle,inner sep=2pt] at (8,0) {};
        \node[fill=black,circle,inner sep=1.5pt]  at (10,-1) {};
        \node[fill=black,circle,inner sep=1.5pt]  at (10,1) {};
        \node[fill=black,rectangle,inner sep=2pt]  at (11,-1.5) {};
        \node[fill=black,rectangle,inner sep=2pt]  at (11,-.5) {};
        \node[fill=black,rectangle,inner sep=2pt]  at (11,.5) {};
        \node[fill=black,rectangle,inner sep=2pt]  at (11,1.5) {};
	\node at (0,1.3) {$\mathcal{T}_{(3,2,1,4)}$};
		
	\draw (8,0)--(11,1.5);
	\draw (8,0)--(11,-1.5);
	\draw (10,-1)--(11,-.5);
	\draw (10,1)--(11,.5);

        \node[fill=black,rectangle,inner sep=2pt]  at (12,-1.5) {};
        \node[fill=black,rectangle,inner sep=2pt]  at (12,-.5) {};
        \node[fill=black,rectangle,inner sep=2pt]  at (12,.5) {};
        \node[fill=black,rectangle,inner sep=2pt]  at (12,1.5) {};
        \node[fill=black,circle,inner sep=1.5pt]  at (13,1) {};
        \node[fill=black,circle,inner sep=1.5pt]  at (13,-1) {};
        \node(B)[fill=black,circle,inner sep=2pt]  at (15,0) {};

	\draw (15,0)--(12,1.5);
	\draw (15,0)--(12,-1.5);
	\draw (13,-1)--(12,-.5);
	\draw (13,1)--(12,.5);
        
        	\draw [dashed] (11,-1.5)--(12,-1.5);
	\draw [dashed] (11,-.5)--(12,.5);
	\draw [dashed] (12,-.5)--(11,.5);
	\draw [dashed] (11,1.5)--(12,1.5);

        \node(C)[fill=black,circle,inner sep=2pt]  at (-1,0) {};
       \node[fill=black,circle,inner sep=1.5pt]  at (0,.5) {};
        \node[fill=black,circle,inner sep=1.5pt]  at (1,0) {};
        \node[fill=black,rectangle,inner sep=2pt]  at (2,1.5) {};
        \node[fill=black,rectangle,inner sep=2pt]  at (2,.5) {};
        \node[fill=black,rectangle,inner sep=2pt]  at (2,-.5) {};
        \node[fill=black,rectangle,inner sep=2pt]  at (2,-1.5) {};

	\draw (-1,0)--(2,-1.5);
	\draw (-1,0)--(2,1.5);
	\draw (0,.5)--(1,0); %
	\draw (1,0)--(2,-.5);
	\draw (1,0)--(2,.5);

        \node[fill=black,rectangle,inner sep=2pt]  at (3,1.5) {};
        \node[fill=black,rectangle,inner sep=2pt]  at (3,.5) {};
        \node[fill=black,rectangle,inner sep=2pt]  at (3,-.5) {};
        \node[fill=black,rectangle,inner sep=2pt]  at (3,-1.5) {};
        \node[fill=black,circle,inner sep=1.5pt]  at (4,0) {};
        \node[fill=black,circle,inner sep=1.5pt]  at (3,.5) {}; %
        \node[fill=black,circle,inner sep=1.5pt]  at (5,.5) {};  
        \node(D)[fill=black,circle,inner sep=2pt]  at (6,0) {};

	\draw (6,0)--(3,-1.5);
	\draw (6,0)--(3,1.5);
	\draw (5,.5)--(3,-.5);
	\draw (4,0)--(3,.5);

	\draw [dashed] (2,1.5)--(3,1.5);
	\draw [dashed] (2,.5)--(3,.5);
	\draw [dashed] (2,-.5)--(3,-1.5);
	\draw [dashed] (2,-1.5)--(3,-.5);
        
\end{tikzpicture}
\end{center}
\caption{The two tanglegrams excluded from planar tanglegrams. The tanglegram on the left
is the catergram $\mathcal{T}_{(3,2,1,4)}$, but the tanglegram on the right is not a catergram, as the trees are not caterpillars.}\label{fig:excluded}
\end{figure}
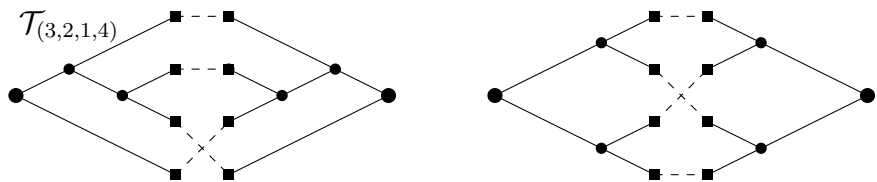

Now we are ready to show:
\begin{proposition} For every $i\in\mathbb{Z}^+$ the catergram $\mathcal{T}_{\rho_i}$ is planar.
\end{proposition}

\begin{proof} Since any leaf-induced subtree of a rooted caterpillar is another caterpillar,
 Theorem~\ref{thm:kurat} yields that $\mathcal{T}_{\rho_i}$ is not planar iff it contains an induced$\mathcal{T}_{(3,2,1,4)}$.
By Proposition~\ref{claim:main}~\ref{part:goal} this happens precisely when one of $(3,2,1,4)$, $(4,2,1,3)$, $(3,2,4,1)$, $(4,2,3,1)$ is a pattern of $\rho_i$.
As $\rho_i$ does not contain a decreasing subsequence of length $3$, $(3,2,1,4)$ and $(4,2,1,3)$ are not among its patterns.
The last entry of the remaining $(3,2,4,1)$ and $(4,2,3,1)$ has three larger elements preceding it, and the first two elements are in decreasing order.  
If  they are patterns of $\rho_i$,
then $1$ must map to either $1$ or $8+2i$. If $1$ maps to $1$, then the other three elements must map to the sequence $(2,3,5)$, and if $1$ maps to $8+2i$, then the remaining three elements must map to a subsequence of
$(9+2i,10+2i,11+2i,12+2i)$. As both of these are increasing,
$(3,2,4,1)$ and $(4,2,3,1)$ are not patterns of $\rho_i$. \qed
\end{proof}

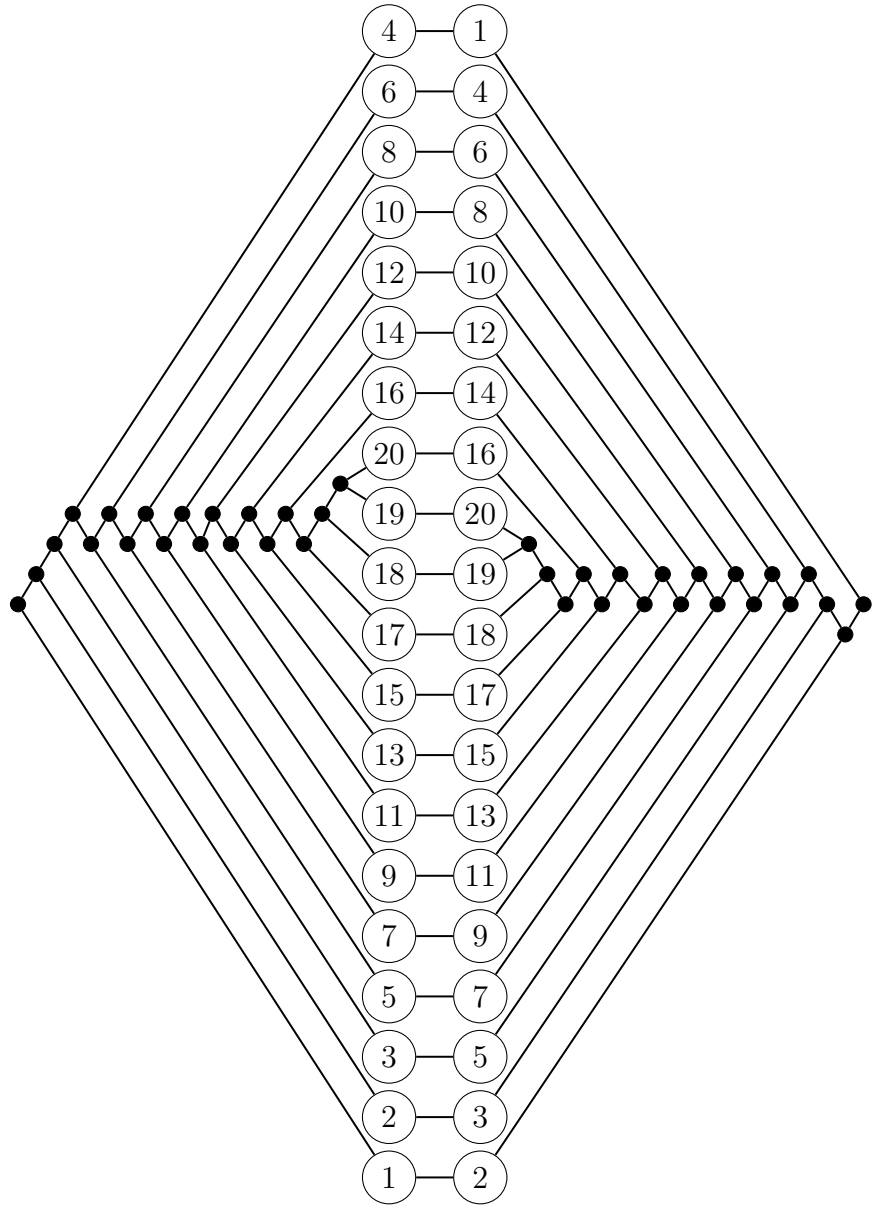
\begin{figure}[htbp]
		\centering
			\begin{tikzpicture}
			[scale=0.8,xscale=1,inner sep=2pt,
			vertex/.style={circle,draw,minimum height=7mm},
			invertex/.style={circle,draw,fill=black,minimum size=1mm},
			thickedge/.style={line width=0.75pt}] 
			\node[vertex] (a1) at (0,1) {$1$};
			\node[vertex] (c2) at (1.5,1) {$2$};			
			\node[vertex] (a2) at (0,2) {$2$};
			\node[vertex] (c3) at (1.5,2) {$3$};			
			\node[vertex] (a3) at (0,3) {$3$};
			\node[vertex] (c5) at (1.5,3) {$5$};			
			\node[vertex] (a4) at (0,20) {$4$};
			\node[vertex] (c1) at (1.5,20) {$1$};			
			\node[vertex] (a5) at (0,4) {$5$};
			\node[vertex] (c7) at (1.5,4) {$7$};					
			\node[vertex] (a6) at (0,19) {$6$};
			\node[vertex] (c4) at (1.5,19) {$4$};				
			\node[vertex] (a7) at (0,5) {$7$};
			\node[vertex] (c9) at (1.5,5) {$9$};					
			\node[vertex] (a8) at (0,18) {$8$};
			\node[vertex] (c6) at (1.5,18) {$6$};				
			\node[vertex] (a9) at (0,6) {$9$};
			\node[vertex] (c11) at (1.5,6) {$11$};						
			\node[vertex] (a10) at (0,17) {$10$};
			\node[vertex] (c8) at (1.5,17) {$8$};					
			\node[vertex] (a11) at (0,7) {$11$};
			\node[vertex] (c13) at (1.5,7) {$13$};						
			\node[vertex] (a12) at (0,16) {$12$};
			\node[vertex] (c10) at (1.5,16) {$10$};						
			\node[vertex] (a13) at (0,8) {$13$};
			\node[vertex] (c15) at (1.5,8) {$15$};						
			\node[vertex] (a14) at (0,15) {$14$};	
			\node[vertex] (c12) at (1.5,15) {$12$};				
			\node[vertex] (a15) at (0,9) {$15$};	
			\node[vertex] (c17) at (1.5,9) {$17$};						
			\node[vertex] (a16) at (0,14) {$16$};	
			\node[vertex] (c14) at (1.5,14) {$14$};					
			\node[vertex] (a17) at (0,10) {$17$};	
			\node[vertex] (c18) at (1.5,10) {$18$};				
			\node[vertex] (a18) at (0,11) {$18$};
			\node[vertex] (c19) at (1.5,11) {$19$};			
			\node[vertex] (a19) at (0,12) {$19$};
			\node[vertex] (c20) at (1.5,12) {$20$};			
			\node[vertex] (a20) at (0,13) {$20$};
			\node[vertex] (c16) at (1.5,13) {$16$};			
			
			\node[invertex] (b1) at (-.8,12.5) {}; 
			\node[invertex] (b2) at (-1.1,12) {};
			\node[invertex] (b3) at (-1.4,11.5) {};
			\node[invertex] (b4) at (-1.7,12) {};
			\node[invertex] (b5) at (-2,11.5) {};
			\node[invertex] (b6) at (-2.3,12) {};
			\node[invertex] (b7) at (-2.6,11.5) {};
			\node[invertex] (b8) at (-2.9,12) {};
			\node[invertex] (b9) at (-3.1,11.5) {};
			\node[invertex] (b10) at (-3.4,12) {};
			\node[invertex] (b11) at (-3.7,11.5) {};
			\node[invertex] (b12) at (-4,12) {};
			\node[invertex] (b13) at (-4.3,11.5) {};
			\node[invertex] (b14) at (-4.6,12) {};
			\node[invertex] (b15) at (-4.9,11.5) {};
			\node[invertex] (b16) at (-5.2,12) {};
			\node[invertex] (b17) at (-5.5,11.5) {};
			\node[invertex] (b18) at (-5.8,11) {};
			\node[invertex] (b19) at (-6.1,10.5) {};

			\node[invertex] (d1) at (2.3,11.5) {}; 
			\node[invertex] (d2) at (2.6,11) {};
			\node[invertex] (d3) at (2.9,10.5) {};
			\node[invertex] (d4) at (3.2,11) {};
			\node[invertex] (d5) at (3.5,10.5) {};
			\node[invertex] (d6) at (3.8,11) {};
			\node[invertex] (d7) at (4.2,10.5) {};
			\node[invertex] (d8) at (4.5,11) {};
			\node[invertex] (d9) at (4.8,10.5) {};
			\node[invertex] (d10) at (5.1,11) {};
			\node[invertex] (d11) at (5.4,10.5) {};
			\node[invertex] (d12) at (5.7,11) {};
			\node[invertex] (d13) at (6,10.5) {};
			\node[invertex] (d14) at (6.3,11) {};
			\node[invertex] (d15) at (6.6,10.5) {};
			\node[invertex] (d16) at (6.9,11) {};
			\node[invertex] (d17) at (7.2,10.5) {};
			\node[invertex] (d18) at (7.5,10) {};
			\node[invertex] (d19) at (7.8,10.5) {};
			
			\draw[thickedge] (b19)--(a1);
			\draw[thickedge] (b18)--(a2);
			\draw[thickedge] (b1)--(b2);
			\draw[thickedge] (b17)--(a3);
			\draw[thickedge] (b3)--(b2);
			\draw[thickedge] (b16)--(a4);
			\draw[thickedge] (b3)--(b4);
			\draw[thickedge] (b15)--(a5);
			\draw[thickedge] (b5)--(b4);
			\draw[thickedge] (b14)--(a6);
			\draw[thickedge] (b5)--(b6);
			\draw[thickedge] (b13)--(a7);
			\draw[thickedge] (b7)--(b6);
			\draw[thickedge] (b12)--(a8);
			\draw[thickedge] (b7)--(b8);
			\draw[thickedge] (b11)--(a9);
			\draw[thickedge] (b9)--(b8);
			\draw[thickedge] (b10)--(a10);
			\draw[thickedge] (b9)--(b10);
			\draw[thickedge] (b9)--(a11);
			\draw[thickedge] (b11)--(b10);
			\draw[thickedge] (b8)--(a12);
			\draw[thickedge] (b11)--(b12);
			\draw[thickedge] (b7)--(a13);
			\draw[thickedge] (b13)--(b12);
			\draw[thickedge] (b6)--(a14);
			\draw[thickedge] (b13)--(b14);
			\draw[thickedge] (b5)--(a15);
			\draw[thickedge] (b15)--(b14);
			\draw[thickedge] (b4)--(a16);
			\draw[thickedge] (b15)--(b16);
			\draw[thickedge] (b3)--(a17);
			\draw[thickedge] (b17)--(b16);
			\draw[thickedge] (b2)--(a18);
			\draw[thickedge] (b17)--(b18);
			\draw[thickedge] (b1)--(a19);
			\draw[thickedge] (b19)--(b18);
			\draw[thickedge] (b1)--(a20);

			\draw[thickedge] (d19)--(c1);
			\draw[thickedge] (d18)--(c2);
			\draw[thickedge] (d1)--(d2);
			\draw[thickedge] (d17)--(c3);
			\draw[thickedge] (d3)--(d2);
			\draw[thickedge] (d16)--(c4);
			\draw[thickedge] (d3)--(d4);
			\draw[thickedge] (d15)--(c5);
			\draw[thickedge] (d5)--(d4);
			\draw[thickedge] (d14)--(c6);
			\draw[thickedge] (d5)--(d6);
			\draw[thickedge] (d13)--(c7);
			\draw[thickedge] (d7)--(d6);
			\draw[thickedge] (d12)--(c8);
			\draw[thickedge] (d7)--(d8);
			\draw[thickedge] (d11)--(c9);
			\draw[thickedge] (d9)--(d8);
			\draw[thickedge] (d10)--(c10);
			\draw[thickedge] (d9)--(d10);
			\draw[thickedge] (d9)--(c11);
			\draw[thickedge] (d11)--(d10);
			\draw[thickedge] (d8)--(c12);
			\draw[thickedge] (d11)--(d12);
			\draw[thickedge] (d7)--(c13);
			\draw[thickedge] (d13)--(d12);
			\draw[thickedge] (d6)--(c14);
			\draw[thickedge] (d13)--(d14);
			\draw[thickedge] (d5)--(c15);
			\draw[thickedge] (d15)--(d14);
			\draw[thickedge] (d4)--(c16);
			\draw[thickedge] (d15)--(d16);
			\draw[thickedge] (d3)--(c17);
			\draw[thickedge] (d17)--(d16);
			\draw[thickedge] (d2)--(c18);
			\draw[thickedge] (d17)--(d18);
			\draw[thickedge] (d1)--(c19);
			\draw[thickedge] (d19)--(d18);
			\draw[thickedge] (d1)--(c20);

			\draw[thickedge] (a1)--(c2);
			\draw[thickedge] (a2)--(c3);
			\draw[thickedge] (a3)--(c5);
			\draw[thickedge] (a4)--(c1);
			\draw[thickedge] (a5)--(c7);
			\draw[thickedge] (a6)--(c4);
			\draw[thickedge] (a7)--(c9);
			\draw[thickedge] (a8)--(c6);
			\draw[thickedge] (a9)--(c11);
			\draw[thickedge] (a10)--(c8);
			\draw[thickedge] (a11)--(c13);
			\draw[thickedge] (a12)--(c10);
			\draw[thickedge] (a13)--(c15);
			\draw[thickedge] (a14)--(c12);
			\draw[thickedge] (a15)--(c17);
			\draw[thickedge] (a16)--(c14);
			\draw[thickedge] (a17)--(c18);
			\draw[thickedge] (a18)--(c19);
			\draw[thickedge] (a19)--(c20);
			\draw[thickedge] (a20)--(c16);
			
			\end{tikzpicture}
	    \caption{A planar drawing of $\mathcal{T}_{\rho_4}$ as described in Proposition~\ref{claim:planardraw}~\ref{part:describe}.}
	\label{fig:planar}
\end{figure}

Just having a proof that $\mathcal{T}_{\rho_i}$ is planar is somewhat unsatisfactory; one  naturally wants to see a planar layout of of this catergram.

First note that given a plane tree $P$ of any rooted binary tree $T$ with $n$ unique labeled leaves, 
the drawing of $P$ gives an ordering $(\ell_1,\ldots,\ell_n)$ of the labels by the order they appear on their line in the drawing.
Moreover, if $v$ is an internal vertex of $T$, then the set of leaves that are descendants of $v$,  i.e. the leaves separated by $v$ from the root, 
 must appear in a contiguous block of $(\ell_1,\ldots,\ell_n)$. It is easy to see that if  $(\ell_1,\ldots,\ell_n)$ is an ordering of the leaf labels
such that for every internal vertex $v$ of $T$ the leaves that are descendants of $v$ appear in a continuous block of $(\ell_1,\ldots,\ell_n)$, then there
is precisely one plane tree $P$ of $T$ that puts the leaves in the order $(\ell_1,\ldots,\ell_n)$ on its line of leaves.

If $v$ is an internal vertex of the caterpillar $C_n$ whose leaves are labeled according to our distance   convention, then there is an $i\in[n]$ such that
the set of leaves that are descendants of $v$ are exactly the leaves labeled with entries that are at least $i$. Therefore a permutation
$(\ell_1,\ldots,\ell_n)\in S_n$ arises  from a plane tree of $C_n$ precisely when for every $i\in[n]$, the entries bigger than $i$ appear only one side (left or right)
of $i$ in $(\ell_1,\ldots,\ell_n)$.

\begin{dmydefs} 
Given a  rooted binary tree $T$ on $n$ leaves, which are labeled by the elements of $[n]$,  we call a permutation $(\ell_1,\ldots,\ell_n)\in S_n$  {\textit consistent} with $T$,
 if for every internal vertex $v$  of $T$, then the set of leaves that are descendants of $v$ appear in a contiguous segment  of $(\ell_1,\ldots,\ell_n)$.
A permutation $(\ell_1,\ldots,\ell_n)$ is {\textit cater-good}, if it is consistent with the distance labeled caterpillar $C_n$  \emph{(see Definition~\ref{convention})}, i.e.
for every $i\in[n]$, the entries bigger than $i$ appear only one side (left or right)
of $i$ in $(\ell_1,\ldots,\ell_n)$.
\end{dmydefs}

\begin{proposition}\label{claim:planardraw}
The following facts are obvious:
\begin{enumerate}[label={\upshape (\alph*)}]
\item The tanglegram $(T_1,T_2,M)$, where the leaves of $T_1$ and $T_2$ are labeled, is planar iff there are permutations
$\pi_1=(a_1,\ldots,a_n)$ and $\pi_2=(b_1,\ldots,b_n)$ of the leaf labels of $T_i$, such that $\pi_i$ is consistent with $T_i$ 
for $i=1,2$,
and $M=\{a_ib_i:i\in[n]\}$.
\item The catergram $T_{\sigma}$ is planar iff there is cater-good a permutation $(a_1,\ldots,a_n)$ such that
$(\sigma(a_1),\ldots,\sigma(a_n))$ is also cater-good. A planar layout is obtained by these permutation, putting leaves in their order on the lines $x=0$ and $x=1$.
\item If a permutation $(c_1,\ldots,c_n)$ of $[n]$ is unimodal, then it is cater-good.
\item\label{part:describe} For every $i\in\mathbb{Z}^+$, a planar drawing of $\mathcal{T}_{\rho_i}$ is given by the permutation $(a_1,\ldots,a_{12+2i})$ where $a_1=4$, $(a_1,a_2,a_3)=(1,2,3)$, $(a_{8+i},a_{9+i},a_{10+i},a_{11+i})=(9+2i,10+2i,11+2i,12+2i)$, and for $j\in[4+i]$,
$a_{3+j}=3+2j$ and $a_{12+2i-j}=4+2j$. 
\end{enumerate}
\end{proposition}

Note that the permutation $(a_1=1,\ldots,a_{12+2i})$ in~\ref{part:describe} is unimodal, and consequently so is $(\rho_i(a_1),\ldots,\rho_i(a_{12+2i}))=(a_2,a_3\ldots,a_{12+2i},1)$.
Figure~\ref{fig:planar} gives the planar drawing of $\mathcal{T}_{\rho_4}$ determined by the permutation given in this Proposition.

\end{document}